\documentclass[useAMS]{tCON2e}

\usepackage{graphicx}
\usepackage[usenames,dvipsnames]{color}
\usepackage{transparent,url}
\usepackage{amsmath,amsfonts,amssymb,dsfont,mathrsfs,amscd,bbm}
\usepackage{arrayjobx}
\usepackage{paralist,subfigure,float,wrapfig,transparent,mathtools,comment}
\usepackage[backgroundcolor=gray,bordercolor=red]{todonotes}
\usepackage{hyperref,bookmark}
\usepackage[all]{hypcap}
\usepackage{tikz,xifthen}
\usetikzlibrary{calc,backgrounds,shapes,arrows,positioning}

\bibpunct[, ]{(}{)}{,}{a}{}{,}
\def\etal{\mbox{\em{et al.}}}

\newcommand{\mb}[1]{\boldsymbol{#1}}

\newcommand{\E}[2]{\mathbb{E}_{#1}[{#2}]}
\newcommand{\EO}[1]{\mathbb{E}_{#1}}
\newcommand{\Prob}[1]{\mathbb{P}\{#1\}}

\newcommand{\qed}{\hfill$\blacksquare$}

\newcommand{\ind}[1]{\mathbbm{1}\{#1\}}
\newcommand{\orth}{\perp\!\!\!\perp}

\newcommand*\rfrac[2]{{}^{#1}\!/_{#2}}

\newcommand\numberthis{\addtocounter{equation}{1}\tag{\theequation}}

\newcommand{\figname}[1]{Figure~#1}

\newcommand{\thname}[1]{Theorem~#1}
\newcommand{\appname}[1]{Appendix~#1}

\newcommand{\eqname}[1]{Equation~#1}

\makeatletter
\renewcommand{\env@cases}[1][@{}l@{\quad}l@{}]{%
  \let\@ifnextchar\new@ifnextchar
  \left\lbrace
  \def\arraystretch{1.5}%
  \array{#1}%
}

\newtheorem{thm}{Theorem}
\newtheorem{lem}{Lemma}

\hypersetup{
    pdftitle={Optimal Policies for Simultaneous Energy Consumption and Ancillary Service Provision for Flexible Loads under Stochastic Prices and No Capacity Reservation Constraint},
    pdfauthor={Mahdi Kefayati and Ross Baldick},
    pdfsubject={Dissertation},
    bookmarksnumbered=true,     
    bookmarksopen=true,         
    bookmarksopenlevel=1, 
    hidelinks=true,   
    hypertexnames=false,   
    pdfstartview=FitH,       
    pdfpagemode=UseBookmarks,    
}

\pgfdeclareimage{car}{car}
\pgfdeclareimage{gear}{gear}
\pgfdeclareimage{cloud}{cloud}
\pgfdeclareimage{plug}{plug}

\begin{document}

 \issn{1366-5820}
\issnp{0020-7179} 
 \jnum{00} \jyear{2014} \jmonth{January}

\markboth{Kefayati, M. \& Baldick, R.}{Optimal Policies for Simultaneous Energy Consumption and AS Provisions}

\title{Optimal Policies for Simultaneous Energy Consumption and Ancillary Service Provision for Flexible Loads under Stochastic Prices and No Capacity Reservation Constraint}

\author{Mahdi Kefayati$^{\ast}$ and Ross Baldick$^{\dag}$
\thanks{The authors are with the ECE department of the University of Texas at Austin, 1 University Station Austin TX 78712.}
\thanks{$^\ast$Corresponding author. Email: \tt kefayati@utexas.edu}
\thanks{$^\dag$Email: \tt baldick@ece.utexas.edu}
}

\maketitle
\begin{abstract}

Flexible loads, i.e. the loads whose power trajectory is not bound to a specific one, constitute a sizable portion of current and future electric demand. This flexibility can be used to improve the performance of the grid, should the right incentives be in place.

In this paper, we consider the optimal decision making problem faced by a flexible load, demanding a certain amount of energy over its availability period, subject to rate constraints. The load is also capable of providing Ancillary Services (AS) by decreasing or increasing its consumption in response to signals from the Independent System Operator. Under arbitrarily distributed and correlated Markovian energy and AS prices, we obtain the optimal policy for minimizing expected total cost, which includes cost of energy and benefits from AS provision, assuming no capacity reservation requirement for AS provision. We also prove that the optimal policy has a multi-threshold form and can be computed, stored and operated efficiently. We further study the effectiveness of our proposed optimal policy and its impact on the grid.

We show that, while optimal simultaneous consumption and AS provision under real-time stochastic prices is achievable with acceptable computational burden, the impact of adopting such real-time pricing schemes on the network might not be as good as suggested by the majority of the existing literature. In fact, we show that such price responsive loads are likely to induce peak-to-average ratios much more than what is observed in the current distribution networks and adversely affect the grid.
\end{abstract}

\section{Introduction\label{sec:intro}}
Smart grids are expected to bring about fundamental changes in terms of information availability to electricity networks and provide bidirectional communication and energy exchange even to the end points of the distribution network. Although it is generally understood that the closed loop of information/energy exchange can improve the overall energy exchange process in many ways (e.g. reducing cost, increasing reliability), models to quantitatively understand the fundamental benefits of the information availability over smart grids are relatively lacking. 

A considerable portion of the current electricity demand is inherently flexible i.e., electric power need not be delivered to the load at a very specific trajectory over time. For these loads, an amount of energy is needed by some (potentially recurring) deadline. In other words, over the operation time of the load, the trajectory of the delivered power only needs to satisfy some constraints instead of exactly following a specific trajectory. Such loads include most heating/cooling systems and electric vehicles (EV). The most recent EIA statistics~\cite{end-use_2009} suggest that such loads comprise more than 50\% of average residential electricity consumption. The prospect of such loads comprising a considerable portion of emerging loads urges us to investigate the benefits obtained from such flexibility utilizing the infrastructure provided by smart grids. 

Many authors have studied the impacts of demand flexibility in various contexts~\cite{turitsyn_robust_2010, turitsyn_smart_2011, mohsenian-rad_optimal_2010, caramanis_management_2009, caramanis_dr_reg_2012, kefayati_efficient_2010, papavasiliou_supplying_2010} and demonstrated the potential benefits of demand flexibility. Considering a residential setting, Mohsenian-Rad \etal~\cite{mohsenian-rad_optimal_2010} propose a price-based load scheduling algorithm for smart grids. They consider a combination of real-time pricing (RTP) and inclining block rates (IBR) for prices. Moreover, to model delay aversion they inflate future prices exponentially. After formulating the deterministic model and proposing an optimal solution to it, they extended their proposed algorithm to the stochastic case by proposing a price estimation method, which would convert stochastic prices to deterministic equivalents for application of their deterministic algorithm. This \emph{certainty equivalent} approach fails to produce optimal responses and hence the performance of their method is bound by the performance of the price estimation technique they have proposed. Caramanis and Foster~\cite{caramanis_management_2009} consider simultaneous energy consumption and AS provision in the context of a load aggregator and propose an approximate stochastic programming approach to solve the intractable dynamic program resulting from their model. Kefayati and Caramanis~\cite{kefayati_efficient_2010} propose a computationally efficient approximate method for solving this problem. Taking a different approach to utilize demand flexibility, Papavasiliou \etal~\cite{papavasiliou_supplying_2010}, proposes coupling flexible loads with intermittent generation, particularly wind, to reduce the uncertainty of net production. A solution based on approximate dynamic programming is proposed and used to analyze the economics of coupling intermittent and flexible resources. The same problem is approached by Neely \etal~\cite{neely_efficient_2010} through Lyapunov optimization to obtain a computationally efficient approximate solution that guarantees \emph{order-wise} delay and cost performance.

The model we study in this work has close connections to storage assets and some of the benefits provided by storage assets can be achieved by flexible loads. Storage assets and their benefits have been studied by various authors~\cite{eyer_energy_2010, kefayati_cdc_2013, korpaas_operation_2003, denholm_value_2009, teleke_control_2009, divya_battery_2009, qin_optimal_2012, secomandi_optimal_2010}. Korpaas~\etal~\cite{korpaas_operation_2003} studied optimal operation of  a storage asset in combination of a wind farm to meet the output schedule using a forecast model. Similarly, Teleke~\cite{teleke_control_2009} studied smoothing strategies for the combined output of a wind farm and a battery storage asset. Denholm and Sioshansi~\cite{denholm_value_2009}, studied the value of storage asset when co-locating wind farms with Compressed Air Energy Storage (CASE). Secomandi~\cite{secomandi_optimal_2010} considered optimal policies for commodity trading with a   capacitated storage asset and concluded that, under some conditions, a multi-threshold policy is optimal. Faghih~\etal~\cite{faghih_economic_2012} examined ramp-constrained storage assets and particularly demonstrated that storage assets can improve price elasticity near average prices. Finally, Kefayati and Baldick~\cite{kefayati_cdc_2013} studied the problem of optimal storage asset operation under real-time prices and proved that under some conditions, a efficiently computable multi-threshold policy, similar to what is presented here, is optimal.

In \cite{kefayati_cdc_2012}, we considered flexible loads responding to real-time prices from the grid as a signal to reflect the more realistic cost of energy consumption. Yet, if the communication infrastructure and the dynamics of the loads allow, their flexibility can be further utilized by the grid. This would be achieved by the loads not only participating in the market for purchasing energy they need, but also for providing ancillary services. Under such scenarios, loads would not only respond to energy prices but also to requests for adjusting their consumption, at rates much faster than market or pricing intervals, which are typically issued by the ISO. Such adjustment signals are typically designed to fine tune supply demand mismatches by regulating system frequency, as we discussed them in \cite{kefayati_allerton_2012}. This form of service, however, needs much faster communication and capability but as we previously discussed, are well within rates available to residential and commercial Internet users. Such ancillary services are typically traded as \emph{capacity}, that is, the operator pays for the ability to adjust the current rate (consumption of generation) of the provider within a specified range (i.e. capacity).

In this paper, we consider an extension of the problem we considered in \cite{kefayati_cdc_2012} by adding the potential for offering ancillary services to the problem of optimal energy consumption by a flexible load. We first examine the model carefully and show that multiple cases are possible depending on the form of capacity constraints to which the load is subjected when changing consumption at paces faster than the decision interval. We then show that our previous techniques to obtain the optimal policy can be extended to joint optimization of energy consumption and AS provision and the optimal policy has a very similar multi-threshold form. We also show that the optimal policy can be computed efficiently in a similar recursive fashion and the complexity order of computations stays the same. Considering joint energy consumption and AS provision is not only a natural extension of the optimal consumption problem, but also play as a bridge towards the work on coordinated energy delivery, where an Energy Services Company (ESCo) is in charge of solving the optimal energy delivery and AS provision for a group of flexible loads. 

The rest of this paper is organized as follows: In Section~\ref{sec:coopt.stoch.model} we present and analyze our model and the optimal demand satisfaction algorithm. 
In sections \ref{sec:coopt.optimal.policy.nores} we present our proposed optimal joint consumption and AS provision policy for the case that AS provision does not need capacity reservation.  Section~\ref{sec:coopt.case.study} is dedicated to analysis of network level performance of our proposed algorithm through simulation as well as its system level impacts. We conclude the paper and comment on our future directions in Section~\ref{sec:coopt.conclusion}. The proof of the theorems are moved to the appendices for better readability.

\section{System Model\label{sec:coopt.stoch.model}}
We model a consumer with a flexible load subject to a total energy demand and energy rate limits. Assuming a discrete time model, our objective is to find out how much energy the consumer should consume towards its total demand so that its total expected cost is minimized. The consumer can also offer some of its flexible capacity for use as reserve in return for some reward.

Let us first introduce the price model. We assume a Markovian price structure for energy and reserve prices, denoted by $ \pi_{t}^{e} $ and $ \pi_{t}^{r} $ respectively. Defining $ \mb{\pi}_{t}\triangleq (\pi_{t}^{e}, \pi_{t}^{r}) $:
\begin{equation}\label{sys:coopt.price.evol}
\begin{array}{l}
\mb{\pi}_{t}=\mb{\lambda}_{t}(\mb{\theta}_{t})+\mb{\epsilon}_t,\qquad \mb{\theta}_{0}\text{ given},\\
\end{array}
\end{equation}
where $\mb{\epsilon}_t\triangleq(\epsilon_{t}^{e},\epsilon_{t}^{r})$ is the random vector capturing price innovations, $\mb{\lambda}_t(\mb{\theta_{t}})\triangleq(\lambda_{t}^{e}(\mb{\theta_{t}}),\lambda_{t}^{r}(\mb{\theta_{t}}))$ models inter-stage correlation of prices and seasonality and $ \mb{\theta}_{t} $ represents the state of the Markov process and throughout this work we assume $  \mb{\theta}_{t} =  \mb{\pi}_{t-1} $, i.e. the previous prices form the state of the price process. We denote the distribution function of $\mb{\epsilon}_t$ by $F_t(\bullet)$ and assume independent price innovations over time, i.e. $ \mb{\epsilon}_{t}\perp\!\!\!\perp\mb{\epsilon}_{t'},\:\forall t\neq t' $. Note that the prices are assumed to be non-stationary, generally distributed and arbitrarily correlated (between energy and reserves). Moreover, we assume $\mb{\lambda}_t(\mb{\theta}_{t})$ to be monotone but otherwise arbitrary to avoid some technicalities. 

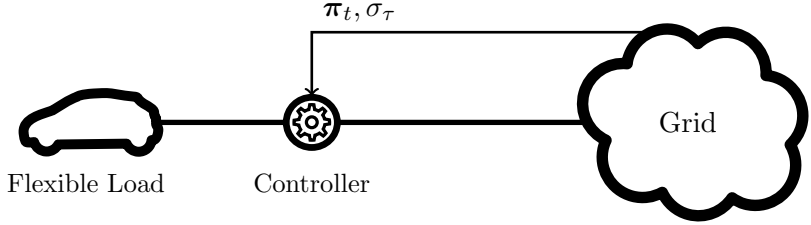
\begin{figure}
\centering
\begin{tikzpicture}
  \draw[white] (-1,-1.5) rectangle (1,2.5);
  \node[minimum size=5em] (ev) at (0,0) {\pgfbox[center,center]{
  	\scalebox{-1}[1]{\includegraphics[width=5em]{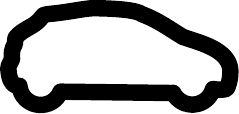}}}};
  \node at ($(ev)+(0,-0.8)$) {\small Flexible Load};
  \node[minimum size=1.8em] (gear) at (3,0) {\pgfbox[center,center]{\includegraphics[width=2em]{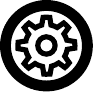}}};
  \node at ($(gear)+(0,-0.8)$) {\small Controller};
  \node[minimum size=7em] (cloud) at (8,0) {\pgfbox[center,center]{
  	\includegraphics[width=8em]{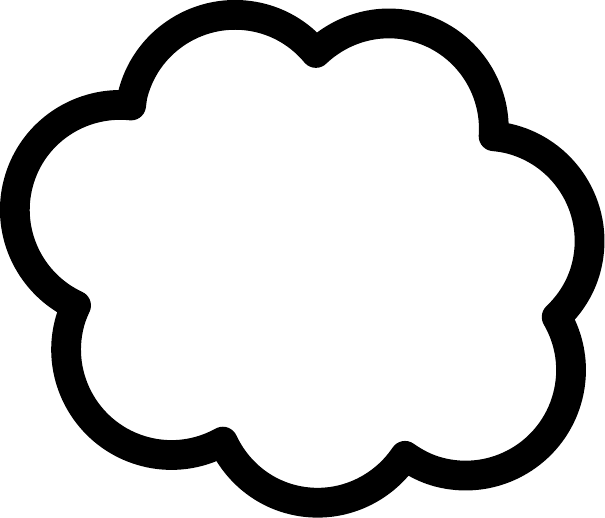}}};
  \node at (cloud) {Grid};
  \draw[-,line width=2pt] (ev) -- (gear);
  \draw[-,line width=2pt] (gear) -- (cloud);
  \draw[<-,line width=1pt] (gear) --  ($(gear)+(0,1.2)$) node[above right]{$\mb{\pi}_{t},\sigma_{\tau}$} -- ($(cloud)+(-0.6,1.2)$);
\end{tikzpicture}
\caption{System model.}
\label{fig:coopt.arch}
\end{figure}

We model the cost minimization problem faced by the consumer as a Dynamic Program (DP) (a.k.a. Markov Decision Process). Under this model, which is schematically depicted in Figure \ref{fig:coopt.arch}, the consumer is assumed to have a certain amount of energy demand at time $ t=0 $, and needs to make decisions about exactly how to consume electricity in the next $T$ time periods with the knowledge of past and current prices and (remaining) energy demand. In other words, the consumer has a deadline of $T$ time slots and seeks an optimal demand satisfaction policy. We assume that the consumer acts as a price taker; consequently, the optimal policy alternatively captures the consumer's bid for purchasing energy and offer for providing reserves. The consumer is also subject to consumption rate limits; that is, its consumption at every stage cannot exceed a certain amount. This models the capacity limits that a typical consumer is facing, from local transformer capacity limits to the rate supported by the device (e.g. charging capacity of a charger). 

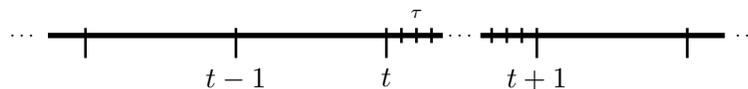
\begin{figure}[b]
\centering
\tikzstyle{main}=[-,line width=2pt]
\tikzstyle{sub}=[-,line width=1pt]
\tikzstyle{mid}=[-,dotted,line width=1pt]
\begin{tikzpicture}
  \draw[main] (0,0) -- (5.25,0);
  \draw[main] (5.75,0) -- (9,0);
  \node at (5.5,0) {\scriptsize\ldots};
  \node[left] at (0,0) {\scriptsize\ldots};
  \node[right] at (9,0) {\scriptsize\ldots};
  
  \foreach \x in {0,...,4}
  	\draw[sub] (0.5+2*\x,0.1) -- (0.5+2*\x,-0.3);
  \node[below] at (2.5,-0.3)  {$t-1$};
  \node[below] at (4.5,-0.3) {$t$};
  \node[below] at (6.5,-0.3) {$t+1$};

  \foreach \x in {1,2,3,7,8,9}
  	\draw[sub] (4.5+0.2*\x,0.1) -- (4.5+0.2*\x,-0.1);
  \node[above] at (4.9,0.1) {\scriptsize $\tau$};
\end{tikzpicture}
\caption{Timeline of decisions versus reserve deployments.}
\label{fig:coopt.tl}
\end{figure}

The consumer is also assumed to sell its flexibility, i.e. its ability to change its consumption rate, as a reserve capacity to the grid too. By modulating its output in a time frame faster than the decision intervals, see Figure \ref{fig:coopt.tl}, it provides a balancing service (a.k.a. reserve) to the grid. An example of such a situation in ERCOT is providing Regulation (REG) service in intervals of four seconds while the market (and hence energy prices) is cleared every five minutes. Since the fast deployments of the reserves is usually designed to manage the uncertainty between the anticipated supply demand balance and its actual behavior, it is fair to expect that reserve deployments have a zero mean and balance out over (relatively) long periods of time. Figure \ref{fig:coopt.sp} demonstrates a sample path of this operation where $e_t$, energy consumption by the flexible load, is modulated by reserve deployments ($ \sigma_\tau $). Note that the reserve deployments happen at about two orders of magnitude higher frequency and hence the time indexes are different. Since the consumer is assumed to be solely capable of consuming energy and not injecting it back, the amount of reserves offered by the consumer cannot exceed its consumption in that time slot. On the upper bound on consumption rate however, the actual constraint on the instantaneous rate depends on the underlying limiting factor and can be either an average rate constraint (over the time slot) or an instantaneous rate limit similar to the lower bound on consumption. For example, if the constraint on the maximum consumption rate is due to thermal limits, like a transformer, then typically the capacity can be modeled as an average rate limit per time slot. Consequently, while instantaneous changes in consumption (due to the summation of designated consumption and reserve deployments) can exceed the upper bound of the consumption rate, the average rate of consumption remains within desired limits since reserve deployments are assumed to be zero mean. This is the model we adopt in this paper.

Alternatively, the rate limit could be due to a hard limit and hence, the total consumption rate (the sum of designated energy consumption rate and reserve deployments) cannot exceed the rate limit. This alternate model basically asks the consumer to reserve some capacity for reserves that could otherwise be used for energy consumption and will result in a more complex optimal policy. We leave this case for future work.

\begin{figure}
\centering
\includegraphics[width=0.7\linewidth]{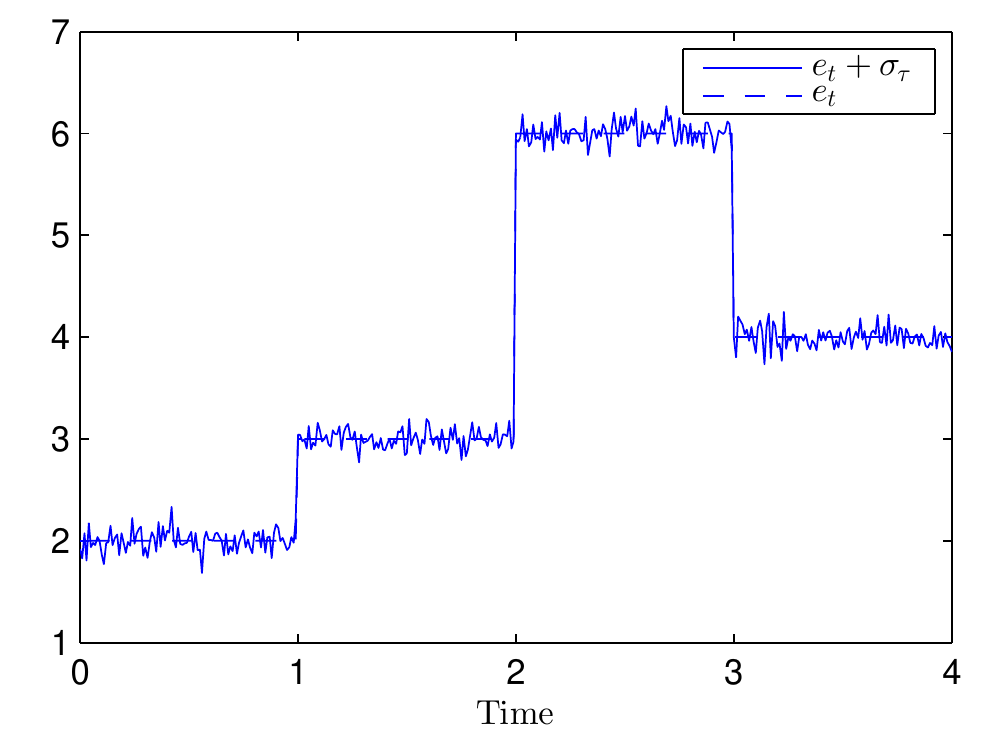}
\caption[Providing reserves while consuming energy, sample path.]{A sample path of energy consumption by the consumer while providing reserves.}
\label{fig:coopt.sp}
\end{figure}

With this model and denoting the remaining energy demand of the consumer by $ d_t $, and the amount of consumed energy and offered reserve by $ e_{t}$ and $ r_{t} $, respectively, we have the following dynamics: 
\begin{equation}\label{sys:coopt.orig}
\begin{array}{l}
d_{t+1}=d_t-e_t+s_t,\qquad d_0 \text{ given},\\
0\leq r_t \leq e_t\leq \min\{d_t,\overline{e}\},\\
\end{array}
\end{equation}
where $ d_0 $ denotes the initial energy demand of the flexible load, $ \overline{e} $ denotes the maximum amount of energy the consumer can consume in one time slot, $ s_t =\sum_{\tau} \sigma_{\tau}$ denotes the sum of reserve deployments over the time interval. Generally, $ s_t $ is assumed to be a zero mean independent process, i.e. $ \E{}{s_{t}}=0,\:\forall t $ and $ \mb{\epsilon}_{t}\orth s_{t}\orth s_{t'},\:\forall t\neq t' $; in this work, however, we assume $s_t=0,\:\forall t$ as a simplifying assumption. That is, we assume that reserve deployments over each time slot are balanced and sum up to zero. Note that in some jurisdictions, there are balanced reserve products that ensure a balanced property. 

The objective of the problem is to minimize the expected cost, i.e. to obtain $ J^*_0(d_0,\mb{\theta}_0) $:
\begin{equation}
J^*_0(d_0,\mb{\theta}_0){=}\min_{\substack{e_t(d_t,\mb{\pi}_t),\\ r_t(d_t,\mb{\pi}_t)}}\E{\mb{\epsilon}_t}{\sum_{t=0}^{T-1} g_t(e_t,r_t,\mb{\pi}_t){+} g_T(d_T)},
\end{equation}
where the minimization is over policies that give $e_t $ and $r_t$ upon observing $ (d_t,\mb{\pi}_t) $, which we denote by $e_t(d_t,\mb{\pi}_t), r_t(d_t,\mb{\pi}_t)$ by abuse of notation. Finally, we need to define our stage and final cost:
\begin{eqnarray}
g_t(e_t,r_t,\mb{\pi}_t)&=& \pi_{t}^{e} e_t-\pi_{t}^{r} r_t,\label{sys:coopt.step.cost}\\
g_T(d_T)&=&\hat{m}_T d_T,\label{sys:coopt.final.cost}
\end{eqnarray}
where $ m_T $ is basically the marginal cost of unsatisfied energy demand which depends on the type of the load. For example, for plug-in hybrid vehicles, $ \hat{m}_T $ is on the order of the equivalent price of gasoline. This model can capture a wide range of flexible loads. The prime example of such loads are battery charging loads like plug-in electric vehicles.

\section{Optimal Energy Consumption and Reserve Provision Policy without Capacity Reservation\label{sec:coopt.optimal.policy.nores}}
First, we present our main result which gives the optimal joint energy consumption and reserve provision policy. Then, we present the algorithm to recursively calculate the parameters of this optimal policy derived from the main theorem and discuss its implications and complexity. Finally, we establish another theorem which considers the result for the uncorrelated price case.

\begin{thm}\label{thm:coopt.uniform.cor}
Consider the system described in~\eqref{sys:coopt.price.evol}--\eqref{sys:coopt.final.cost}.
\begin{enumerate}
\renewcommand{\theenumi}{(\alph{enumi})}
\renewcommand{\labelenumi}{\theenumi}
\item The optimal value function is continuous, piecewise linear and convex with $T+1$ pieces given by:
\begin{equation}\label{eq:coopt.uniform.cap.val}
\begin{split}
J^*_0(d_0,\mb{\theta}_0)=&\sum_{j=0}^{T-1} m^{j}_0(\mb{\theta}_0)[(d_0-j\overline{e})^+{\wedge}\overline{e}]+m^{T}_0(\mb{\theta}_0)(d_0-T\overline{e})^{+},
\end{split}
\end{equation}
where $a\wedge b \triangleq \min\{a,b\}$ and $m^{i}_t(\mb{\theta}_t)$ is given by the following backward recursion:
\begin{equation}\label{eq:coopt.rec.uniform.cap}
m^{i}_t(\mb{\theta}_t)=\E{\mb{\epsilon}_t}{M_i(\mb{\theta}_t,\mb{\epsilon}_t)},
\end{equation}
where,
\begin{equation}\label{eq:coopt.def.uniform.M}
M_i(\mb{\theta}_t,\mb{\epsilon}_t){=}\begin{cases}[@{}l@{\qquad}r@{}l@{}]
m^{i}_{t+1}(\mb{\pi}_t) & \hat{m}^{i}_{t+1}\leq& \pi^a_t,\\
\pi^a_t & \hat{m}^{i-1}_{t+1}\leq& \pi^a_t < \hat{m}^{i}_{t{+}1},\\
m^{i-1}_{t+1}(\mb{\pi}_t) & &\pi^a_t < \hat{m}^{i-1}_{t+1},
\end{cases}
\end{equation}
$\mb{\pi}_t=\mb{\lambda}_t(\mb{\theta}_t)+\mb{\epsilon}_t$ by \eqref{sys:coopt.price.evol}, $\pi^a_t=\pi^e_t-(\pi^r_t)^+$, $m^{i}_T(\mb{\theta}_t)=\hat{m}^{i}_T=\hat{m}_T,\:\forall i$, $m^{0}_t(\mb{\theta}_t)=\hat{m}^{0}_t=-\infty,\:\forall t$  and $\hat{m}^{i}_t$ is given by:
\begin{equation}\label{eq:coopt.rec.m.eq}
\hat{m}^{i}_t=\{m^{i}_t(\mu_1,\mu_2)| m^{i}_t(\mu_1,\mu_2)=\mu_1-(\mu_2)^+\}.
\end{equation}

\item The optimal policy is given by:
\begin{eqnarray}
e^*_{t}(d,\mb{\pi}_{t})&=& (d-i^*\overline{e})^+ \wedge \overline{e},\label{eq:coopt.uniform.cap.policy.e}\\
r^*_{t}(d,\mb{\pi}_{t})&=& e^*_{t}(d,\mb{\pi}_{t})\ind{\pi_{t}^{r}>=0},\label{eq:coopt.uniform.cap.policy.r}
\end{eqnarray}
where $ \ind{\bullet} $ is the indicator function and,
\begin{equation}i^*=\max\{i|\hat{m}^{i}_{t}<\pi_{t}^{e}-(\pi_{t}^{r})^{+}\}.\end{equation}
\end{enumerate}
\end{thm}
The proof is provided in Appendix \ref{app:coopt.pf.uniform.cor}.

The value function obtained in Theorem \ref{thm:coopt.uniform.cor} shows an interesting property of the problem under study: even with arbitrary prices and correlation, the value function remains not only piecewise linear, but also, all the pieces have the same length, namely $\overline{e}$. More importantly, the number of pieces scales only linearly with the number of available time slots, $T$, as opposed to exponential scaling which typically happens in dynamic programming problems \cite{bertsekas_dynamic_2005}. Consequently, a relatively simple multi-threshold optimal policy for joint demand satisfaction and reserve provision can be obtained. Since such a policy is either to be implemented in embedded devices such as the EVSE, PEV, Home Energy Management System (HEMS) or thermostat or used to control a large group of flexible loads by a load aggregator, such scalability is very important. 

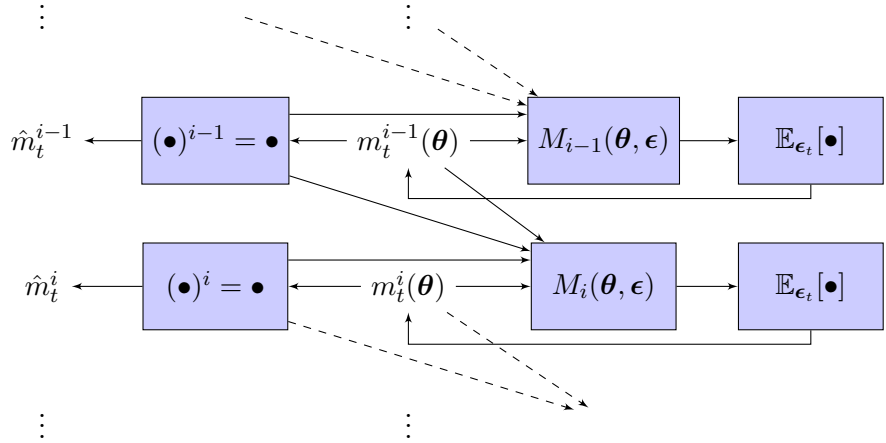
\begin{figure}
\centering
\tikzstyle{block} = [draw, fill=blue!20, rectangle, minimum height=3em, minimum width=5em]
\tikzstyle{signal} = [minimum width=2em]
\begin{tikzpicture}[auto, node distance=2em,>=latex',
                    skip loop/.style={to path={-- ++(0,#1) -| (\tikztotarget)}},
                    jump/.style={to path={-- ++(0,#1) -| ($(\tikztotarget)+(-1.5,-#1)$) -- ($(\tikztotarget.north)+(0,-#1)$) -- (\tikztotarget)}}
                    ]
    \matrix[row sep=2em,column sep=2em] {
    \node [signal] {\vdots};  & \node [signal] (fpu){}; & \node [signal] (restu){\vdots}; & & \\
    \node [signal] (mi1){$\hat{m}^{i-1}_{t}$}; &
    \node [block] (fp1) {$ (\bullet)^{i-1}=\bullet $}; &
    \node [signal] (mit1) {$ m^{i-1}_{t}(\mb{\theta}) $}; &
    \node [block] (M1) {$ M_{i-1}(\mb{\theta},\mb{\epsilon}) $}; & 
    \node [block] (E1) {$ \E{\mb{\epsilon}_t}{\bullet} $}; \\
    \node [signal] (mi){$\hat{m}^{i}_{t}$}; &
    \node [block] (fp) {$ (\bullet)^{i}=\bullet $}; &
    \node [signal] (mit) {$ m^{i}_{t}(\mb{\theta}) $}; &
    \node [block] (M) {$ M_i(\mb{\theta},\mb{\epsilon}) $}; &
    \node [block] (E) {$ \E{\mb{\epsilon}_t}{\bullet} $}; \\       
    \node [signal] {\vdots};  & & \node [signal] {\vdots}; & \node [signal] (restd) {}; & \\
    };
    \draw [->] (fp) -- (mi);
    \draw [->] (fp1) -- (mi1);
    \draw [->] (mit) -- (fp);
    \draw [->] (mit1) -- (fp1);
    \draw [->] (mit) -- (M);
    \draw [->] (mit1) -- (M1);
    \draw [->] (M) -- (E);
    \draw [->] (M1) -- (E1);
    \draw [->] (mit1) -- (M);    
    \draw [->] ($(fp.north east)!0.4!(fp.east)$) -- ($(M.north west)!0.4!(M.west)$);
    \draw [->] ($(fp1.south east)!0.2!(fp1.east)$) -- ($(M.north west)!0.2!(M.west)$);
    \draw [->] ($(fp1.north east)!0.4!(fp1.east)$) -- ($(M1.north west)!0.4!(M1.west)$);
    \draw [->,dashed] (restu) -- (M1);
    \draw [->,dashed] (mit) -- (restd);
    \draw [->,dashed] ($(fp.south east)!0.2!(fp.east)$) -- ($(restd.north west)!0.2!(restd.west)$);
    \draw [->,dashed] ($(fpu.south east)!0.2!(fpu.east)$) -- ($(M1.north west)!0.2!(M1.west)$);
    
    \path (E1) edge [->,skip loop=-2em] (mit1)
          (E) edge [->,skip loop=-2em] (mit);
\end{tikzpicture}
\caption[Recursive calculation of thresholds.]{Recursive calculation of value function coefficients and thresholds in block diagram form. Note that recursions are backward on $t$ and $(\bullet)^{j}=\bullet$ represents the corresponding solution to \eqref{eq:coopt.rec.m.eq}.}
\label{fig:coopt.blockdiagram.cor}
\end{figure}

While Theorem \ref{thm:coopt.uniform.cor} gives a closed form optimal policy and value function, it also encapsulates the essential richness of the problem due to the general price structure into the coefficients of the value function and the corresponding thresholds obtained through \eqref{eq:coopt.rec.m.eq}. Whether a closed form can be attained for these thresholds depends on the assumed price statistics. Nevertheless, Theorem \ref{thm:coopt.uniform.cor} gives a straightforward algorithm for calculating these thresholds. Figure \ref{fig:coopt.blockdiagram.cor} depicts the recursive algorithm that is used to calculate the coefficients and corresponding thresholds in a block diagram form, mainly around the $i$th element. This block diagram basically depicts \eqref{eq:coopt.rec.uniform.cap}, \eqref{eq:coopt.def.uniform.M} and \eqref{eq:coopt.rec.m.eq}. Let us go through these steps for more clarity, assuming time $t$ at the beginning:
\begin{enumerate}
\item Using $m^{j}_{t}(\mb{\theta})$, obtain $\hat{m}^{j}_{t}$ using \eqref{eq:coopt.rec.m.eq} for all $j$.
\item For all $j$ form the corresponding $M_j(\mb{\theta},\mb{\epsilon})$ using $m^{j}_{t}(\mb{\theta})$, $m^{j-1}_{t}(\mb{\theta})$, $\hat{m}^{j}_{t}$, $\hat{m}^{j-1}_{t}$ obtained in previous step and \eqref{eq:coopt.def.uniform.M}.
\item For all $j$ take the expected value of $M_j(\mb{\theta},\mb{\epsilon})$ and obtain $m^{j}_{t-1}(\mb{\theta})$ as in \eqref{eq:coopt.rec.uniform.cap}.
\item Repeat these steps letting $t=t-1$.
\end{enumerate}

A considerable advantage of the above algorithm for obtaining the optimal thresholds is that it can be implemented in parallel very efficiently. In particular, at each time $t$, all the $T$ pieces, indexed by $j$ can be calculated in parallel. This is implicitly reflected also in the steps described above, noting that each step happens for all $j$ simultaneously for a given time slot $ t $, without using any of the information corresponding to other time slots. This parallelism can also be seen in the parallel branches of the block diagram in Figure \ref{fig:coopt.blockdiagram.cor}.

The computational complexity of obtaining the optimal policy using Theorem \ref{thm:coopt.uniform.cor} is $O(\frac{T^2}{\delta})$, assuming the operations in equations \eqref{eq:coopt.rec.uniform.cap}, \eqref{eq:coopt.def.uniform.M} and \eqref{eq:coopt.rec.m.eq} are $ O(1) $ and the resolution of $m^{j}_{t}(\mb{\theta})$ in $ \mb{\theta} $ is $ O(\delta) $. Note that $m^{j}_{t}(\mb{\theta})$ is a function of $ \mb{\theta} $ and, hence, at worst, it needs to be calculated and stored numerically. Given the above discussion on the parallel computation of the optimal policy, establishing this bound is straight forward. Note that we have at most $O(T)$ pieces and $T$ time slots. Bear in mind that assuming $ O(1) $ computation time on each operation only implies that these computations do not depend on $T$ and hence, do not scale with the problem.

In terms of storage complexity of the optimal policy, Theorem \ref{thm:coopt.uniform.cor} gives an even better result. Note that once \eqref{eq:coopt.rec.m.eq} is solved, there is no need to store $m^{j}_{t}(\mb{\theta})$ since $\hat{m}^{j}_{t}$ is enough to run the optimal policy. Therefore, the storage complexity of the optimal policy is $O(T^2)$, which is great for embedded systems or when the optimal policy is computed in the cloud and should be transferred to the controller.

Another advantage of the proposed computational structure becomes vivid when multiple flexible loads are involved and an aggregator is controlling all the loads simultaneously. In such setups, the load aggregator can reuse the calculated coefficients and thresholds for the loads which share the same (absolute) deadline and capacity. To observe this property, first note that the loads with the same per stage capacity have the same break points in their value function. Now, consider two loads with potentially different demands, say $ d $ and $ d' $ and potentially different dwell lengths, e.g. $ T $ and $ T' $, but with the same deadline, i.e. $t_{d}=t'_{d}$. Now, for calculating the optimal coefficients backward, the actual statistics seen by these loads are the same for all times between $t_{d}=t'_{d}$ and $ t_{d}-\min\{T,T'\} $ since their deadlines are equal. Note that since the times are measured relative to arrival times, the time indexes for these two loads might not be the same. For the period where both loads share the same statistics, i.e. between $t_{d}-\min\{T,T'\}$ and $ t_{d} $, the coefficients obtained in backward recursion discussed in the above algorithm are the same. Note that the coefficients, $m^{j}_{t}(\mb{\theta})$, do not depend on the demand. Hence, the load aggregator essentially needs to calculate these coefficients and optimal thresholds only once per deadline and load capacity, for the maximum amount of load dwell time. Moreover, in case a new load arrives sharing deadline and capacity with another load but staying for more time, computations are needed to be performed only for the extra amount of steps this load stays compared to longest staying load with the same deadline and capacity. That is, the optimal policy parameters only get augmented by the ones that are not calculated before arrival of a new load.

\begin{figure}
\centering
\tikzstyle{int}=[draw, minimum size=2em]
\tikzstyle{pt}=[]
\tikzstyle{plt}=[-,line width=1pt]
\tikzstyle{gd}=[-,dashed,line width=1pt]
\tikzstyle{gdy}=[-,dotted,line width=1pt]
\tikzstyle{nd}=[fill]
\begin{tikzpicture}[
	scale=1, transform shape,
	every node/.style={inner sep=0pt}]
\node [pt] (ptx0) at (-2,0) {};
\node [pt] (ptx1) at (1.5,0) {};
\node [pt] (ptx2) at (3,0) {};
\node [pt] (ptx3) at (7,0) {};
\node [pt] (ptx4) at (11,0) {};
\node [pt] (pty0) at (0,-1) {};
\node [pt] (pty1) at (0,-1) {};
\node [pt] (pty2) at (0,-1) {};
\node [pt] (pty3) at (0,2) {};
\node [pt] (pty4) at (0,3) {};
\node [pt] (pty5) at (0,5) {};
%
\path [fill=green!10,fill opacity=0.8,draw opacity=0] ($(ptx2)+(pty0)$) -- ($(ptx2)+(pty5)$) -- ($(ptx3)+(pty5)$) -- ($(ptx3)+(pty0)$) -- cycle;
\path [fill=blue!30,fill opacity=0.8,draw opacity=0] ($(ptx0)+(pty0)$) -- ($(ptx0)+(pty5)$) -- ($(ptx2)+(pty5)$) -- ($(ptx2)+(pty0)$) -- cycle;
\node [below left, inner sep=3pt, color=blue!50!black] at ($(ptx2)+(pty5)$) {Consume};
\node [below, inner sep=3pt, color=green!50!black] at ($(ptx2)!0.5!(ptx3)+(pty5)$) {Partially consume};
\node [below right, inner sep=3pt, color=black] at ($(ptx3)+(pty5)$) {Do not consume};
\draw [->] (ptx0) -- (ptx4);
\draw [->] (pty0) -- (pty5);
\node [below left, inner sep=3pt] at (pty5) {$e^*$};
\node [below left, inner sep=3pt] at (ptx4) {$\pi_t^{a}=\pi_t^{e}-(\pi_t^{r})^{+}$};
\draw [plt] ($(ptx0)+(pty4)$) -- ($(ptx2)+(pty4)$);
\draw [plt] ($(ptx2)+(pty3)$) -- ($(ptx3)+(pty3)$);
\draw [plt] ($(ptx3)$) -- ($(ptx4)$);
\draw [gd]  ($(pty3)$) -- ($(ptx2)+(pty3)$);
\draw [gdy]  ($(ptx2)+(pty4)$) -- ($(ptx2)$);
\draw [gdy]  ($(ptx3)+(pty3)$) -- ($(ptx3)$);
\draw [nd] ($(ptx2)+(pty3)$) circle [radius=1.5pt];
\draw [nd] ($(ptx3)+(pty3)$) circle [radius=1.5pt];
\draw [nd] ($(ptx2)+(pty4)$) circle [radius=1.5pt];
\draw [nd] (ptx2) circle [radius=1.5pt];
\draw [nd] (ptx3) circle [radius=1.5pt];
\node [left, inner sep=3pt] at (pty3) {$d_{t}{-}i\overline{e}$};
\node [above left, inner sep=3pt] at (pty4) {$\overline{e}$};
\node [below, inner sep=3pt] at (ptx2) {$\hat{m}^{i-1}_{t}$};
\node [below, inner sep=3pt] at (ptx3) {$\hat{m}^{i}_{t}$};
\end{tikzpicture}
\caption{Optimal policy.}
\label{fig:coopt.opt.policy}
\end{figure}
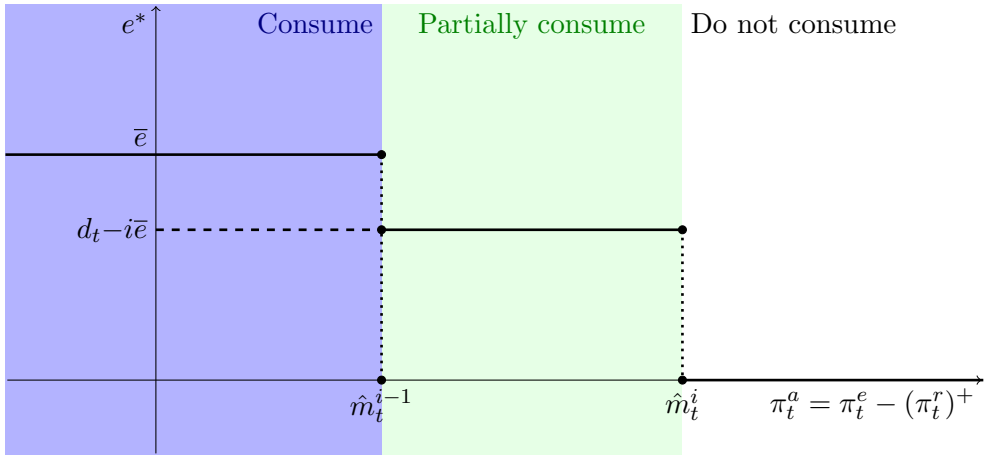

The optimal coefficients in the value function and their corresponding thresholds also have a very interesting economic interpretation. Consider a load with demand $d$ and capacity $\overline{e}$ and let us define $ i=\lfloor\frac{d}{\overline{e}}\rfloor$, then for any $t$, $m^{i}_{t}(\mb{\theta})$ basically gives the expected effective marginal cost of energy for this load, noting that the expected value function is the expected cost to go for such a load. This economic interpretation of the optimal coefficients results in a very intuitive interpretation of the optimal policy: It basically says consume if the effective price, $ \pi^{a}_{t} $, i.e. energy price minus reserve price if positive, is better than what you believe as the expected effective marginal cost of consumption (at the level of remaining demand upon finishing consumption). If the effective price is only better than the current expected effective marginal cost to go, then only consume enough to satisfy the partial demand. Finally, if the effective price is higher than the current expected effective marginal cost, then do not consume. \figname{\ref{fig:coopt.opt.policy}} depicts this interpretation. Note that the price axis in this figure, and the price considered in the optimal policy, is the effective price which depends on price of energy as well as price of reserve. \figname{\ref{fig:coopt.price.regions}} depicts price regions corresponding to the three effective price regions in \figname{\ref{fig:coopt.opt.policy}} more vividly. Note that these regions are exactly the ones defined by the conditional function $ M_i(\mb{\theta},\mb{\epsilon}) $ in \eqref{eq:coopt.def.uniform.M}.


\begin{figure}
\centering
\tikzstyle{int}=[draw, minimum size=2em]
\tikzstyle{pt}=[minimum size=0em]
\tikzstyle{plt}=[-,line width=1pt]
\tikzstyle{gd}=[-,dashed,line width=1pt]
\tikzstyle{gdy}=[-,dotted,line width=1pt]
\tikzstyle{nd}=[fill]
\begin{tikzpicture}[
	scale=1, transform shape,
	every node/.style={inner sep=0pt}]
\node [pt] (ptx0) at (-2,0) {};
\node [pt] (ptx1) at (2,0) {};
\node [pt] (ptx2) at (4,0) {};
\node [pt] (ptx4) at (7,0) {};
\node [pt] (pty0) at (0,-1) {};
\node [pt] (pty1) at (0,2) {};
\node [pt] (pty2) at (0,4) {};
\node [pt] (pty5) at (0,5) {};

\path [fill=blue!40,fill opacity=0.8,draw opacity=0] ($(ptx0)+(pty0)$) -- ($(ptx0)$) -- ($(ptx1)$)-- ($(ptx1)+(pty0)$)  -- cycle;
\path [fill=blue!50,fill opacity=0.8,draw opacity=0] ($(ptx0)$) -- ($(ptx1)!2!(pty1)$) -- ($(ptx1)$) -- cycle;
\path [fill=green!30,fill opacity=0.8,draw opacity=0] ($(ptx1)!2!(pty1)$) -- ($(ptx0)+(pty5)$) -- ($(ptx2)!1.25!(pty2)$) -- ($(ptx2)$) --  ($(ptx1)$) -- cycle;
\path [fill=green!20,fill opacity=0.8,draw opacity=0] ($(ptx2)$) -- ($(ptx2)+(pty0)$) -- ($(ptx1)+(pty0)$) -- ($(ptx1)$) -- cycle;
\path [fill=black!10,fill opacity=0.8,draw opacity=0] ($(ptx2)$) -- ($(ptx2)+(pty0)$) -- ($(ptx4)+(pty0)$) -- ($(ptx4)$) -- cycle;
\draw [->] (ptx0) -- (ptx4);
\draw [->] (pty0) -- (pty5);
\node [below left, inner sep=3pt] at (pty5) {$\pi_{t}^{r}$};
\node [below left, inner sep=3pt] at (ptx4) {$\pi_{t}^{e}$};
\draw [plt] ($(ptx1)+(pty0)$) -- ($(ptx1)$) -- ($(ptx1)!2!(pty1)$);
\draw [plt] ($(ptx2)+(pty0)$) -- ($(ptx2)$) -- ($(ptx2)!1.25!(pty2)$);
\draw [nd] (ptx1) circle [radius=1.5pt];
\draw [nd] (ptx2) circle [radius=1.5pt];
\draw [nd] (pty1) circle [radius=1.5pt];
\draw [nd] (pty2) circle [radius=1.5pt];
\node [above right, inner sep=3pt] at (ptx1) {$\hat{m}^{i-1}_{t}$};
\node [above right, inner sep=3pt] at (ptx2) {$\hat{m}^{i}_{t}$};
\node [above right, inner sep=3pt] at (pty1) {$\hat{m}^{i-1}_{t}$};
\node [above right, inner sep=3pt] at (pty2) {$\hat{m}^{i}_{t}$};
\end{tikzpicture}
\caption{Price regions.}
\label{fig:coopt.price.regions}
\end{figure}
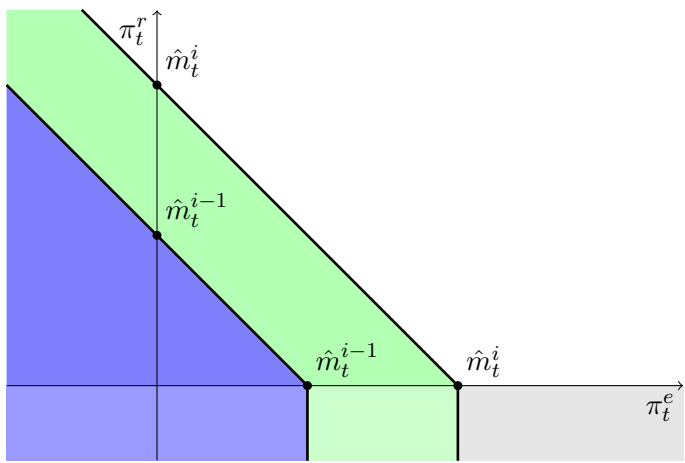

If further assumptions are made about the price statistics, the computational complexity of the optimal policy can be further improved. An interesting case in this direction is the independent price case, which also covers deterministic prices. 
We address the independent case in the following theorem:

\begin{thm}\label{thm:coopt.uniform.indep}
Consider the system described in~\eqref{sys:coopt.price.evol}--\eqref{sys:coopt.final.cost} and assume $\mb{\lambda}_t(\mb{\theta})=0,\: \forall t,\mb{\theta}$, then, the optimal value function given in Theorem~\ref{thm:coopt.uniform.cor} simplifies to:
\begin{equation}\label{eq:coopt.uniform.cap.val.indep}
J^*_0(d_0)=\sum_{j=0}^{T-1} \hat{m}^{j}_0[(d_0-j\overline{e})^+{\wedge}\overline{e}]+\hat{m}^{T}_0(d_0-T\overline{e})^{+},
\end{equation}
where $m^{i}_t$ is given by:
\begin{equation}\label{eq:coopt.rec.m.eq.indep}
\hat{m}^{i}_t=\hat{m}^{i}_{t+1} - G_{t}(\hat{m}^{i-1}_{t+1},\hat{m}^{i}_{t+1}),
\end{equation}
in which, 
\begin{equation}\label{eq:coopt.def.G2}
G_t(z,z')\triangleq\int_{z}^{z'} F^a_t(\zeta)\,d\zeta,
\end{equation}
where $ F^a_t(\bullet) $ is the marginal probability distribution function of the effective price random variable defined as $ \pi^a_t=\epsilon^a_t\triangleq \epsilon_t^{e}-(\epsilon_t^{r})^{+}$.

Moreover, the optimal policy is given by given by the same policy as in Theorem~\ref{thm:coopt.uniform.cor}.
\end{thm}
The proof is provided in Appendix \ref{app:coopt.pf.uniform.indep}.

A corollary of Theorem~\ref{thm:coopt.uniform.indep}, is that the computational and storage complexity of the description of the optimal policy and value function is $\Theta(T)$ for each $t$, and hence $\Theta(T^2)$. This is assuming that evaluating $G_t(\bullet,\bullet)$ function is $\Theta(1)$, which is in line with other assumptions we had in the general case.

Similar to the general case, this solution can be implemented in $T$ parallel processes in a very straightforward way essentially resulting in a $ (T+1) \times (T+1) $ table that represents the optimal thresholds for every $i$ and $t$. The implementation complexity of each branch, however, is reduced dramatically. Figure \ref{fig:coopt.blockdiagram.indep} depicts the block diagram form of the optimal threshold calculation algorithm.

As discussed, the independent case covers the deterministic price case. In deterministic case finding the optimal consumption and reserve provision amounts is a straightforward optimization problem, Theorem \ref{thm:coopt.uniform.indep} gives a quick, recursive and parallel method to obtain the optimal policy independent of the actual amount of demand. This is particularly useful for an aggregator which would calculate the policy once and reuse it for many loads as discussed in the general case. Note that in deterministic case, $G_t(\bullet,\bullet)$ can be calculated in closed form very easily as:
\begin{equation}
G_t(z,z')=(z'-(z\vee\epsilon^a_t))^+,
\end{equation}
where $ (a\vee b) \triangleq \max\{a,b\} $ and $ \epsilon^a_t=\epsilon^e_t-(\epsilon^r_t)^+ $ is a given deterministic number. 

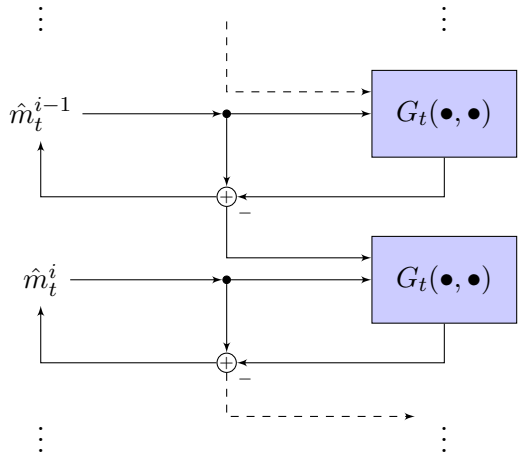
\begin{figure}
\centering
\tikzstyle{block} = [draw, fill=blue!20, rectangle, minimum height=3em, minimum width=5em]
\tikzstyle{signal} = [minimum width=2em]
\tikzstyle{anot} = [minimum width=0em, inner sep=0mm, outer sep=0mm]
\tikzstyle{sum} = [draw,circle,inner sep=0mm,minimum size=2mm]
\tikzstyle{dot} = [draw,circle,inner sep=0mm,minimum size=1mm, fill=black]
\begin{tikzpicture}[auto, node distance=2em,>=latex',
                    skip loop/.style={to path={-- ++(0,#1) -| (\tikztotarget)}},
                    jump/.style={to path={-- ++(0,#1) -| ($(\tikztotarget)+(-1.5,-#1)$) -- ($(\tikztotarget.north)+(0,-#1)$) -- (\tikztotarget)}}
                    ]
    \matrix[row sep=1em,column sep=4em] {
    \node [signal] {\vdots}; & \node [signal] (restu){}; & \node [signal] {\vdots}; \\
    \node [signal] (mi1){$\hat{m}^{i-1}_{t}$}; &
    \node [dot] (d1) {};&
    \node [block] (E1) {$ G_t(\bullet,\bullet)$}; \\
    &
    \node [sum] (s1) {\tiny $+$};&
    \\
    \node [signal] (mi){$\hat{m}^{i}_{t}$}; &
    \node [dot] (d) {};&
    \node [block] (E) {$ G_t(\bullet,\bullet) $}; \\       
    &
    \node [sum] (s) {\tiny $+$};&
    \\       
    \node [signal] {\vdots}; & & \node [signal] (restd) {\vdots}; \\
    };
    \draw [->] (mi) -- (d);
    \draw [->] (d) -- (E);
    \draw [->] (d) -- (s);
    \draw [->] (E) |- node[pos=0.98] {\tiny $-$} (s);
    \draw [->] (s) -| (mi);
    \draw [->] (mi1) -- (d1);
    \draw [->] (d1) -- (E1);
    \draw [->] (d1) -- (s1);
    \draw [->] (E1) |- node[pos=0.98] {\tiny $-$} (s1);
    \draw [->] (s1) -| (mi1);   
    \draw [->] (s1) |- ($(E.north west)!0.5!(E.west)$);
    \draw [->,dashed] (restu) |- ($(E1.north west)!0.5!(E1.west)$);
    \draw [->,dashed] (s) |- ($(restd.north west)!0.5!(restd.west)$);    
\end{tikzpicture}
\caption[Recursive calculation of thresholds, independent price case.]{Recursive calculation of optimal thresholds under price independence assumption as described by \eqref{eq:coopt.rec.m.eq.indep}.}
\label{fig:coopt.blockdiagram.indep}
\end{figure}


\section{Numerical Analysis\label{sec:coopt.case.study}}
Although we have proven the optimality of the proposed algorithms mathematically, we still need to establish the improvements of optimal response and also compare it to the no AS case we studied in \cite{kefayati_cdc_2012}. To this end, we use a similar setup as in \cite{kefayati_cdc_2012} where we studied the cost improvements for PEV loads based on the dataset and charging session patterns obtained in \cite{kefayati_allerton_2012}.

In a same setup as in \cite{kefayati_cdc_2012}, we considered 10,000 scenarios in which a group of 1000 PHEVs which show up over a 24 hour period of time. Energy demand, arrival and departure patterns are based on the results of our study in \cite{kefayati_allerton_2012} using the transportation dataset published by NREL \cite{nrel_transportation}, assuming a minimum dwell time of three hours and availability of charging everywhere under a non-anticipative model \cite{kefayati_itec_2014}. Dwell times of the loads are truncated to 24 hour for ease of calculations.

Energy price statistics are based on the real-time market prices in the Houston Load Zone for year 2012 and only independent case is considered for simplicity. The prices are assumed to be normally distributed with mean from average Houston Load Zone for each hour and simulations are done for various price uncertainties reflected in the standard deviation of price realizations, denoted by $\sigma$. Price uncertainty here can be be interpreted as Gaussian price estimation error from the flexible loads perspective. AS prices are based on average of ERCOT REG Up and REG Down prices and are assumed to be known to the load because AS prices are typically obtained in the day-ahead market and hence are available at real-time.

For the comparative cost performance study, we have depicted the normalized per charging session costs in \figname{\ref{fig:coopt.cost.comp}}, normalized by the no-AS case. The comparison  is done against no-AS optimal algorithm, the static price responsive CEC based method we introduced in \cite{kefayati_cdc_2012}, as well as price oblivious methods we studied in \cite{kefayati_allerton_2012}: immediate and AR charging. Moreover, to compare the network level impact of AS providing optimal response, we have plotted the average diurnal pattern of load produced by the above mentioned charging policies in \figname{\ref{fig:coopt.load.comp}}.

\begin{figure}[t]
\centering
\includegraphics[width=0.99\linewidth]{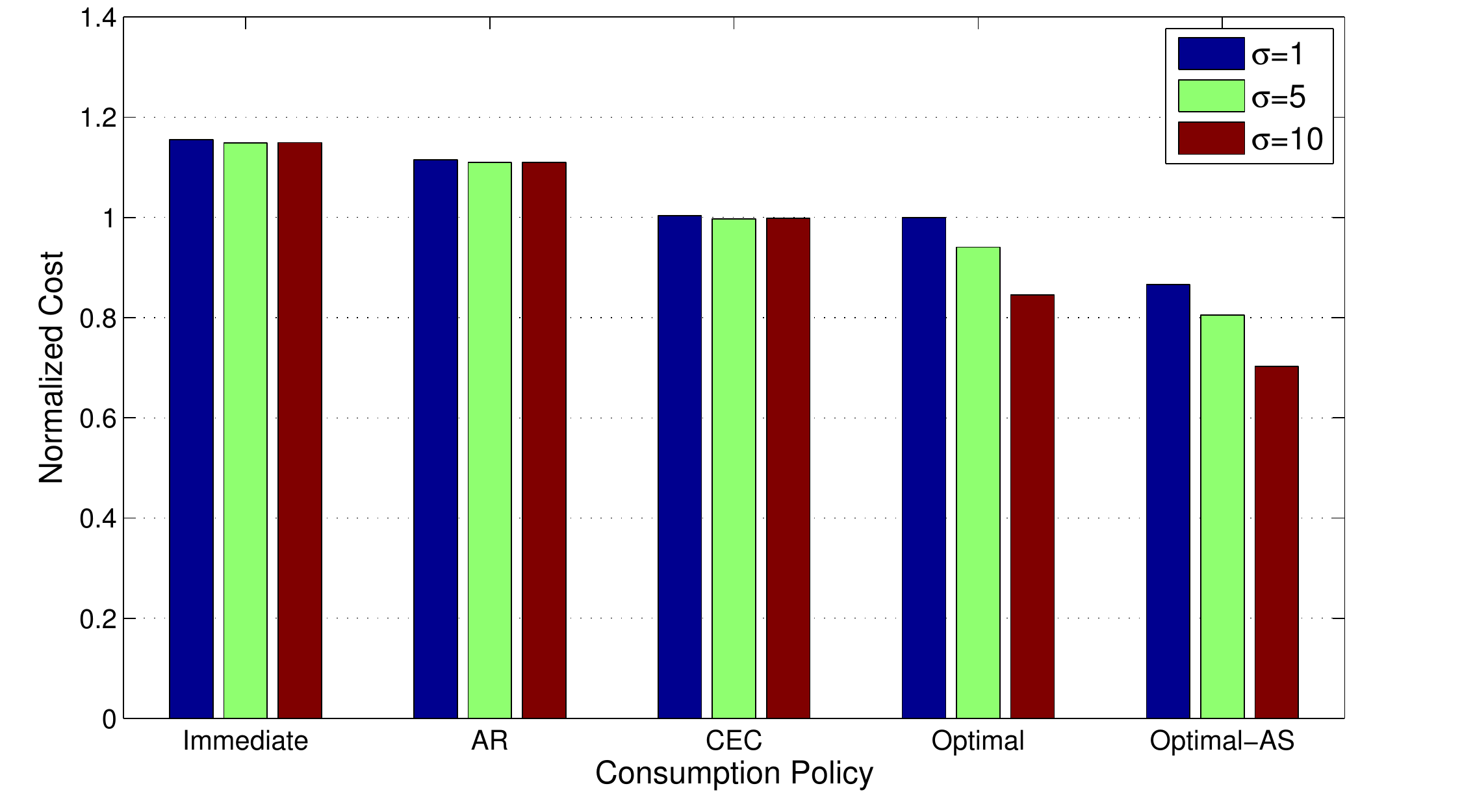}
\caption[Cost comparison, optimal with AS vs. others.]{Cost comparison between various consumption policies under different levels of uncertainty.}
\label{fig:coopt.cost.comp}
\end{figure}

\figname{\ref{fig:coopt.cost.comp}} demonstrates considerable reduction in total cost due to AS provision, roughly about 10\% to 15\%, for all uncertainty levels. Since this cost savings is accompanied by AS capacity offered to the grid equal to the amount of load served (since AS prices are always positive in practice and hence optimal amount of AS to offer is equal to the amount of consumption due to \thname{\ref{thm:coopt.uniform.cor}} and \eqname{\ref{eq:coopt.uniform.cap.policy.r}}), it is in the best interest of both load and the grid to adopt this method versus the optimal consumption, if possible. On the other hand, the added cost of infrastructure for receiving AS commands and reporting back to the grid, or localized AS provision in case of autonomous response, should be considered. Finally, \figname{\ref{fig:coopt.load.comp}} compares the average diurnal pattern of load induced by various consumption policies. Based on this figure, we observe no significant change in the pattern of load, mainly because the pattern of AS prices is very similar to the pattern of energy prices as depicted in the figure.

\begin{figure}[t]
\centering
\includegraphics[width=0.99\linewidth]{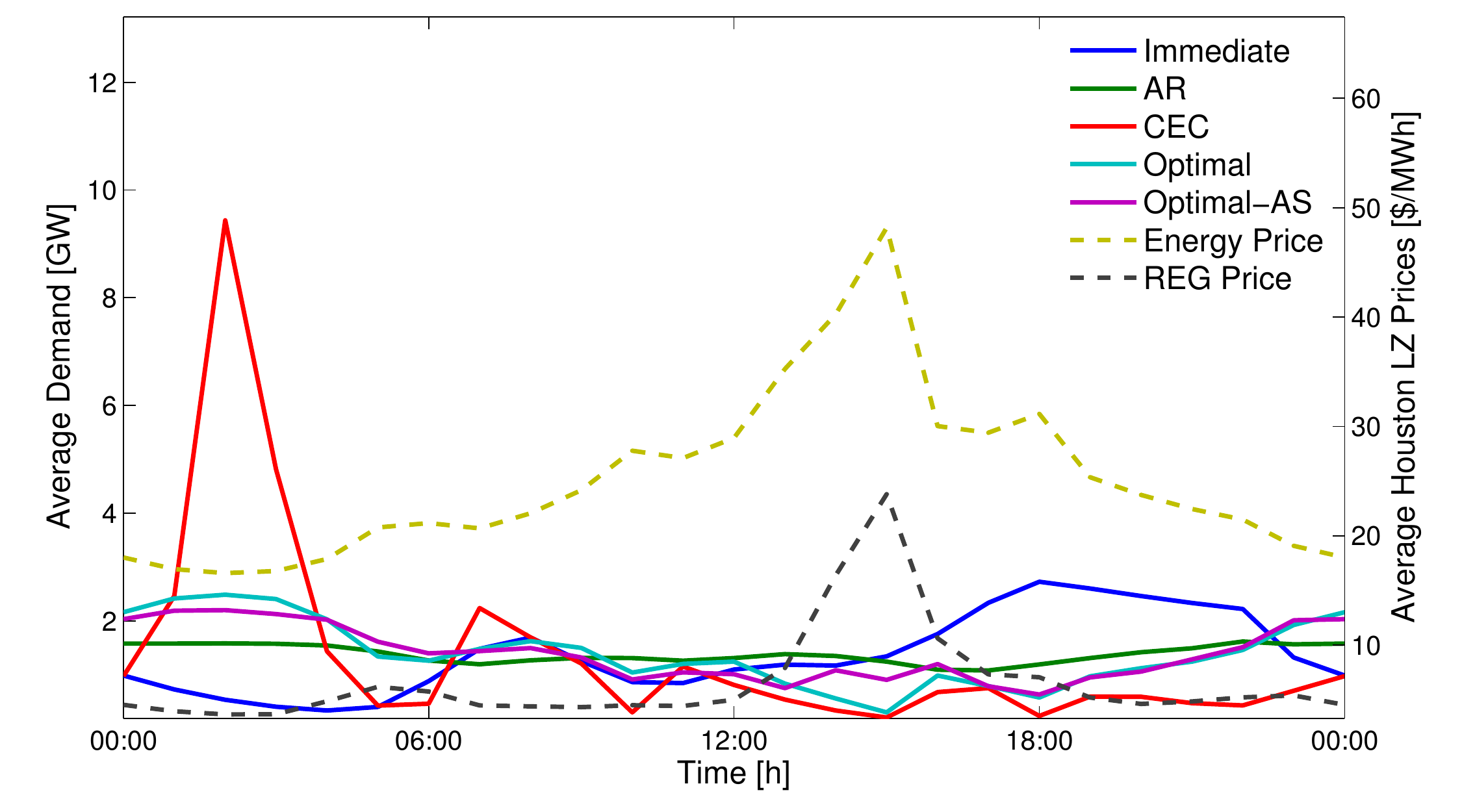}
\caption[Average diurnal load comparison, optimal with AS vs. others.]{Average diurnal load comparison between various consumption policies ($ \sigma=10 $).}
\label{fig:coopt.load.comp}
\end{figure}

\section{Conclusion\label{sec:coopt.conclusion}}
In this paper we extended our results in the \cite{kefayati_cdc_2012} for optimal response of flexible loads to the case where the loads can provide ancillary services in sub-market interval time frame and showed that a similar optimal policy structure holds. We particularly considered the case where no capacity reservation is needed for AS provision. Considering the case with capacity reservation is an ongoing future work.

Based on our performance evaluation on PEV load, we observed similar network load patters due to the high correlation of energy and AS prices. Moreover, we observed consistent cost reduction across the studied levels of uncertainty. Combined with the fact that AS provided by the flexible loads can ultimately help the grid, we conclude that AS providing optimal response should be adopted when communication infrastructure is readily available or economically justifiable.   

As a natural direction for this work, we are already working on the case with capacity reservation requirement. Our preliminary results suggest that the same optimal multi-threshold optimal policy structure holds and there exists a similar efficient method for calculating the thresholds. We plan to perfect this case as our future work. To extend this work further, we also consider AS providing optimal response as building block for approximating the coordinated energy delivery problem we studied in \cite{kefayati_efficient_2010} and \cite{kefayati_transaction_pricing_2011}. This is another direction we are planning to work on in the future.

\section*{Acknowledgments}
This work has been supported by EV-STS and NSF.
\bibliographystyle{tCON}
\bibliography{energy}

\appendices
\section{Proof of Theorem~\ref{thm:coopt.uniform.cor}\label{app:coopt.pf.uniform.cor}}
We establish that the proposed optimal value function satisfies the Bellman equation. We assume the CDFs and expected values exist throughout and the correlation structure is well behaved, which holds for most practical cases. Before going through the details, let us first establish some lemmas that help us streamline the proof. 

\begin{lem}\label{lem:coopt.fact}
For any $d$, $e$, $\overline{e}$ such that $0\leq d$, $0\leq\overline{e}$, $0\leq e \leq \overline{e}$, and letting $i=\lfloor\rfrac{d}{\overline{e}}\rfloor$, and $\tilde{d}=d-i\overline{e} $, we have:
\begin{enumerate}
\item $ [(d-e-(i-1)\overline{e})^+\wedge\overline{e}]=\overline{e}-(e-\tilde{d})^+$,
\item $(\tilde{d}{-}e)^+=\tilde{d}-(e\wedge\tilde{d})$.
\end{enumerate}
\end{lem}
\begin{proof}
\begin{enumerate}
\item First observe that the left hand side (LHS) can be simplified as:
$$ [(d-e-(i-1)\overline{e})^+\wedge\overline{e}] = [(\tilde{d}-e+\overline{e})^+\wedge\overline{e}].$$
Considering the left hand side, only three cases are possible:
\begin{enumerate}
\item $ \tilde{d}-e+\overline{e} \leq 0$: In this case, LHS=0. This condition can be rearranged as $\overline{e} \leq e-\tilde{d}$. But given the conditions in the assumption, this can only happen if $\overline{e} = e-\tilde{d}$. Consequently, the right hand side (RHS) results in:
$$ \overline{e}-(e-\tilde{d})^+ = \overline{e} - \overline{e} = 0.$$
\item $ 0 \leq \tilde{d}-e+\overline{e} \leq \overline{e}$: In this case, LHS=$\tilde{d}-e+\overline{e}$. Rearranging this condition results in $0 \leq e-\tilde{d}\leq \overline{e}$, which results in:
$$ \overline{e}-(e-\tilde{d})^+ = \overline{e} - e + \tilde{d} = \text{LHS}.$$
\item $ \overline{e} \leq \tilde{d}-e+\overline{e}$: In this case, LHS=$\overline{e}$. This case can be also rearranged to $ e-\tilde{d} \leq 0$, and hence:
$$ \overline{e}-(e-\tilde{d})^+ = \overline{e} - 0 = \text{LHS}.$$
\end{enumerate}
\item For this part, only two cases can happen, if $ \tilde{d}{-}e \geq 0 $, then:
$$(\tilde{d}{-}e)^+=\tilde{d}-e= \tilde{d}-(e\wedge\tilde{d}),$$
otherwise,
$$(\tilde{d}{-}e)^+=0 = \tilde{d}-\tilde{d} = \tilde{d}-(e\wedge\tilde{d}).$$
\end{enumerate}

\end{proof}

\begin{lem}\label{lem:coopt.e.split}
For any  $\tilde{d}$, $e$ such that $0\leq e$, $0\leq\tilde{d}$,  we have $e= (e-\tilde{d})^+ +(e\wedge\tilde{d})$.
\end{lem}
\begin{proof}
If $e \leq \tilde{d}$:
$$ (e-\tilde{d})^+ +(e\wedge\tilde{d}) = 0 + e = e,$$
otherwise:
$$ (e-\tilde{d})^+ +(e\wedge\tilde{d}) = e-\tilde{d} + \tilde{d} = e.$$
\end{proof}

We skip establishing the convexity of the optimal value function for brevity. As in the proof of \thname{\ref{thm:optresp.uniform.cor}} in \appname{\ref{app:optresp.pf.uniform.cor}}, a proof similar to Proposition 1 in \cite{papavasiliou_supplying_2010} can be applied using convexity of the stage and final costs and linearity of dynamics.

The main proof is based on backward induction on $t$; to this end, let us assume that the proposed following form of the value function in Theorem \ref{thm:coopt.uniform.cor} holds:
\begin{equation}
\begin{split}\label{eq:coopt.val.proof.gen}
J^*_{t}(d_{t},\mb{\theta}_{t})=&\sum_{j=1}^{\infty} m^{j}_{t}(\mb{\theta}_{t})[(d_{t}-j\overline{e})^+\wedge\overline{e}]. 
\end{split}
\end{equation}
where $m^{j}_{t}(\mb{\theta})=m^{T}_t$ for $j\geq T$. We need to show that if this assumption holds for $t+1$, then it holds for $t$. These consecutive time slots are linked by the Bellman equation; that is:
\begin{equation}
\label{eq:coopt.belman.pf}
J^*_t(d_t,\mb{\theta}_t)=\E{\mb{\epsilon}_{t}}{\min_{e,r}\left\{\pi_{t}^{e} e - \pi_{t}^{r} r + J^*_{t+1}(d_t-e,\mb{\pi}_t)\right\}},
\end{equation}
using system dynamics equations and $\mb{\theta}_{t+1}=\mb{\pi}_t$ as assumed previously. Define $i=\lfloor\rfrac{d_t}{\overline{e}}\rfloor$, $i\overline{e}\leq d_t < (i+1)\overline{e}$. Our objective is to obtain $m^{i}_{t}(\mb{\theta})$ in the desired optimal value function; which in this case can be rewritten as:
\begin{equation}\label{eq:coopt.val.in.proof}
\begin{split}
J^*_{t}(d_t,\mb{\theta}_{t})=&\sum_{j=1}^{i-1} m^{j}_{t}(\mb{\theta}_{t})\overline{e}+m^{i}_{t}(\mb{\theta}_{t})\tilde{d}_t,
\end{split}
\end{equation}
where $ \tilde{d}_t\triangleq d_t-i\overline{e} $. Let us first rearrange the $ J^*_{t+1}(d_{t}-e,\mb{\pi}_t) $ term in \eqref{eq:coopt.belman.pf}. Since $0\leq e\leq \overline{e}$, $(i-1)\overline{e}\leq d_t-e\leq (i+1)\overline{e}$. Using \eqref{eq:coopt.val.proof.gen} for $t+1$, by the induction hypothesis and similar to \eqref{eq:coopt.val.in.proof}, we have:
\begin{equation}
\begin{split}
J^*_{t+1}(d_{t}-e,\mb{\pi}_t)&=\sum_{j=1}^{i-2} m^{j}_{t+1}(\mb{\pi}_t)\overline{e}\\
                   &\quad + m^{i-1}_{t+1}(\mb{\pi}_t)[(d_{t}-e-(i-1)\overline{e})^+\wedge\overline{e}]\\
                   &\quad + m^{i}_{t+1}(\mb{\pi}_t)(d_{t}-e-i\overline{e})^+,\\
                   &=\sum_{j=1}^{i-1} m^{j}_{t+1}(\mb{\pi}_t)\overline{e}+m^{i}_{t+1}(\mb{\pi}_t)\tilde{d}_t\\
                   &\quad-m^{i-1}_{t+1}(\mb{\pi}_t)(e-\tilde{d}_t)^+ - m^{i}_{t+1}(\mb{\pi}_t)(e\wedge\tilde{d}_t),
\end{split}\label{eq:coopt.Jt+1}
\end{equation}
where we have used Lemma \ref{lem:coopt.fact} to obtain the second equality. Now, let us substitute \eqref{eq:coopt.Jt+1} in \eqref{eq:coopt.belman.pf} as depicted in \eqref{eq:coopt.pf.steps}: 
\begin{align*}\numberthis\label{eq:coopt.pf.steps}
J^*_t(d_t,\mb{\theta}_t)&=\EO{\mb{\epsilon}_{t}}\left[\min_{e,r}\mathrlap{\left\{\pi_t^e e - \pi_t^r r + \sum_{j=1}^{i-1} m^{j}_{t+1}(\mb{\pi}_t)\overline{e} + m^{i}_{t+1}(\mb{\pi}_t)\tilde{d}_{t}\right.}\right.&\\
&&\left.\left.\vphantom{\sum_{j=1}^{i-1}} - m^{i-1}_{t+1}(\mb{\pi}_t)(e - \tilde{d}_{t})^+ - m^{i}_{t+1}(\mb{\pi}_t)[e\wedge\tilde{d}_{t}]\right\}\right]\displaybreak[0]\\
&\stackrel{(a)}{=}\EO{\mb{\epsilon}_{t}}\left[\min_{e,r}\left\{\vphantom{\sum_{j=1}^{i-1}}\mathrlap{\pi_t^e e - \pi_t^r r} \right.\right.\\
&& \left.\left.\vphantom{\sum_{j=1}^{i-1}}- m^{i-1}_{t+1}(\mb{\pi}_t)(e - \tilde{d}_{t})^+ - m^{i}_{t+1}(\mb{\pi}_t)(e\wedge\tilde{d}_{t})\right\}\right.\quad\\
&&\left. + \sum_{j=1}^{i-1} m^{j}_{t+1}(\mb{\pi}_t)\overline{e} + m^{i}_{t+1}(\mb{\pi}_t)\tilde{d}_{t}\right]\displaybreak[0]\\
&\stackrel{(b)}{=}\EO{\mb{\epsilon}_{t}}\left[\min_{e,r}\left\{\vphantom{\sum_{j=1}^{i-1}}\mathrlap{(\pi_t^e- m^{i-1}_{t+1}(\mb{\pi}_t))(e - \tilde{d}_{t})^+}\right.\right.\\
&&\left.\left.\vphantom{\sum_{j=1}^{i-1}} + (\pi_t^e - m^{i}_{t+1}(\mb{\pi}_t))(e\wedge\tilde{d}_{t}) - \pi_t^r r \right\}\right.\quad\\
&& \left. + \sum_{j=1}^{i-1} m^{j}_{t+1}(\mb{\pi}_t)\overline{e} + m^{i}_{t+1}(\mb{\pi}_t)\tilde{d}_{t}\right]\displaybreak[0]\\
&\stackrel{(c)}{=}\EO{\mb{\epsilon}_{t}}\left[\min_{e}\left\{\vphantom{\sum_{j=1}^{i-1}}\mathrlap{(\pi_t^e - (\pi_t^r)^+ - m^{i-1}_{t+1}(\mb{\pi}_t)) (e - \tilde{d}_{t})^+ }\right.\right.\\
&&\left.\left. \vphantom{\sum_{j=1}^{i-1}}+ (\pi_t^e- (\pi_t^r)^+ - m^{i}_{t+1}(\mb{\pi}_t))(e\wedge\tilde{d}_{t})\right\}\right.\quad\\
&&\left. + \sum_{j=1}^{i-1} m^{j}_{t+1}(\mb{\pi}_t)\overline{e} + m^{i}_{t+1}(\mb{\pi}_t)\tilde{d}_{t}\right]\\
\end{align*}
where, (a) is obtained by rearranging terms, (b) is obtained by using Lemma \ref{lem:coopt.e.split} and factoring common terms. To obtain (c), notice that the minimization problem in $r$ can be tackled directly, that is, if its coefficient, ($\pi_t^r$) is positive, then we want to maximize $r$ and minimize it otherwise. But we already know the $0\leq r\leq e$, hence, we can obtain the optimal reserve offering policy as: 
\begin{equation}
r^*_{t}(d,\mb{\pi}_{t})= e^*_{t}(d,\mb{\pi}_{t})*\ind{\pi_{t}^{r} \geq 0}.\label{eq:coopt.uniform.cap.policy.r.proof}
\end{equation}
Now (c) is obtained by using \eqref{eq:coopt.uniform.cap.policy.r.proof}, essentially substituting $r$ with $e$ with proper conditionals on $\pi_t^r$, namely its positivity. The result of \eqref{eq:coopt.pf.steps} leaves us with a much simpler optimization problem since it is all in terms of $e$, which is essentially cut into two pieces: $(e - \tilde{d}_{t})^+$ and $(e\wedge\tilde{d}_{t})$ using Lemma \ref{lem:coopt.e.split}. The problem at hand is essentially a linear programing problem, however, using the convexity of the value function, we can parametrically solve it, basically by inspection. Since the value function is piecewise linear, its convexity is equivalent to  $m^{j-1}_{t+1}(\mb{\pi}_t)\leq m^{j}_{t+1}(\mb{\pi}_t), \forall j$; noting that $m^{j}_{t+1}(\mb{\pi}_t)$ is the slope of the $j$th piece of the value function. This leaves us with essentially three cases to consider for the minimization problem at hand, which we denote them by events $\mathcal{E}_1$, $\mathcal{E}_2$ and $\mathcal{E}_3$:
\begin{eqnarray}
\mathcal{E}_1: & m^{i}_{t+1}(\mb{\pi}_t) \leq& \pi_t^e - (\pi_t^r)^+,\label{eq:coopt.ev1}\\
\mathcal{E}_2: & m^{i-1}_{t+1}(\mb{\pi}_t) \leq& \pi_t^e - (\pi_t^r)^+ \leq m^{i}_{t+1}(\mb{\pi}_t),\label{eq:coopt.ev2}\\
\mathcal{E}_3: && \pi_t^e - (\pi_t^r)^+ \leq m^{i-1}_{t+1}(\mb{\pi}_t).\label{eq:coopt.ev3}
\end{eqnarray}
Note that in each of these conditions, $\mb{\pi}_t$ appears on both sides and hence the above conditions are implicitly defined. Moreover, notice that each of these events consist of two simple events corresponding to positivity of $\pi_t^r$. Conditioned on each of these events, it is straightforward to solve the final minimization problem in \eqref{eq:coopt.pf.steps}, understanding that the only constraint we are facing is $0\leq e\leq \overline{e}$: Under $\mathcal{E}_1$, none of the coefficients in are positive and hence, the optimal decision is to minimize $e$. Under $\mathcal{E}_2$, only the second coefficient is negative and hence the optimal decision is to maximize the second term, i.e. $e^*=\tilde{d}_t$. Finally, under $\mathcal{E}_3$, both coefficients are negative and hence the optimal decision is to maximize $e$, i.e. $e^*=\overline{e}$. This basically gives us the optimal policy as: 
\begin{equation}\label{eq:coopt.optimal.policy.short}
e^*(d_t,\mb{\pi}_t){=}\begin{cases}[@{}l@{\qquad}l@{:\quad}r@{}l@{}]
0                 & \text{if } (\mathcal{E}_1) & m^{i}_{t+1}(\mb{\pi}_t) \leq& \pi_t^e - (\pi_t^r)^+,\\
d{-}i\overline{e}        & \text{if } (\mathcal{E}_2) & m^{i-1}_{t+1}(\mb{\pi}_t) \leq& \pi_t^e - (\pi_t^r)^+ \leq m^{i}_{t+1}(\mb{\pi}_t),\\
\overline{e}           & \text{if } (\mathcal{E}_3) &&\pi_t^e - (\pi_t^r)^+ \leq m^{i-1}_{t+1}(\mb{\pi}_t).
\end{cases}
\end{equation}
As mentioned before, the conditions defining these events are implicit. Therefore, we need to solve the following equation to obtain the explicit conditions:
\begin{equation}
m_{t+1}^{j}(\pi_t^e,\pi_t^r)=\pi_t^e - (\pi_t^r)^+\label{eq:coopt.fpf.pf}
\end{equation}
Defining $\hat{m}_{t+1}^{j}$ as:
\begin{equation}
\hat{m}_{t+1}^{j}=\{m_{t+1}^{j}(\pi_t^e,\pi_t^r)| m_{t+1}^{j}(\pi_t^e,\pi_t^r)=\pi_t^e - (\pi_t^r)^+\},\label{eq:coopt.fp.pf}
\end{equation}
similar to \eqref{eq:coopt.rec.m.eq}, we can make these conditions explicit. The desired form in \eqref{eq:coopt.uniform.cap.policy.e}, is then obtained by paying attention to the fact that $\hat{m}^{j}_{t+1}$ is increasing in $j$ by convexity and hence there exists $i^*$ such that $\pi_t^e - (\pi_t^r)^+ < \hat{m}_{t+1}^{i^*},\: \forall j>i^* $ and therefore, we can formulate the optimal policy as stated in the theorem.

Plugging in the optimal policy in \eqref{eq:coopt.pf.steps}, we can continue the proof. Conditioning based on the $ \mathcal{E}_k $ events, we have:
\begin{equation}\label{eq:coopt.J.rec.pf}
J^*_t(d_t,\mb{\theta}_t){=}J^*_t(d_t,\mb{\theta}_t|\mathcal{E}_1)\Prob{\mathcal{E}_1}+J^*_t(d_t,\mb{\theta}_t|\mathcal{E}_2)\Prob{\mathcal{E}_2}+J^*_t(d_t,\mb{\theta}_t|\mathcal{E}_3)\Prob{\mathcal{E}_3},
\end{equation}
where,
\begin{equation*}
\begin{split}
J^*_t(d_t,\mb{\theta}_t|\mathcal{E}_1)&{=}\sum_{j=1}^{i-1} \E{\mb{\epsilon}_t}{m^{j}_{t+1}(\mb{\pi}_t)|\mathcal{E}_1}\overline{e} {+}\E{\mb{\epsilon}_t}{m^{i}_{t+1}(\mb{\pi}_t)|\mathcal{E}_1}\tilde{d}_t,\\
J^*_t(d_t,\mb{\theta}_t|\mathcal{E}_2)&{=}\sum_{j=1}^{i-1} \E{\mb{\epsilon}_t}{m^{j}_{t+1}(\mb{\pi}_t)|\mathcal{E}_2}\overline{e} {+}\E{\mb{\epsilon}_t}{\pi_t^e - (\pi_t^r)^+|\mathcal{E}_2}\tilde{d}_t,\\
J^*_t(d_t,\mb{\theta}_t|\mathcal{E}_3)&{=}\sum_{j=1}^{i-2} \E{\mb{\epsilon}_t}{m^{j}_{t+1}(\mb{\pi}_t)|\mathcal{E}_3}\overline{e} {+}\E{\mb{\epsilon}_t}{\pi_t^e - (\pi_t^r)^+|\mathcal{E}_3}\overline{e} {+} \E{\mb{\epsilon}_t}{m^{i-1}_{t+1}(\mb{\pi}_t)|\mathcal{E}_3}\tilde{d}_t.
\end{split}
\end{equation*}

It is now clear that the optimal value function has the desired form and what remains is to calculate the coefficient of $\tilde{d}_t$, i.e. $m^{i}_{t}(\mb{\theta}_t)$, to obtain the full recursion and conclude the proof. To this end, we use \eqref{eq:coopt.J.rec.pf} and combine the three cases we introduced previously:
\begin{equation}\label{eq:coopt.rec.m.cor.pf}
\begin{split}
m^{i}_{t}(\mb{\theta}_t)&{=}\E{\mb{\epsilon}_t}{m^{i}_{t+1}(\mb{\pi}_t)|\mathcal{E}_1}\Prob{\mathcal{E}_1} {+}\E{\mb{\epsilon}_t}{\pi_t^e - (\pi_t^r)^+|\mathcal{E}_2}\Prob{\mathcal{E}_2} {+}\E{\mb{\epsilon}_t}{m^{i-1}_{t+1}(\mb{\pi}_t)|\mathcal{E}_3}\Prob{\mathcal{E}_3}\\
                   &{=} \E{\mb{\epsilon}_t}{m^{i}_{t+1}(\mb{\lambda}_t(\mb{\theta}_t)+\mb{\epsilon}_t)|\mathcal{E}_1}\Prob{\mathcal{E}_1}\\ 
                   &\quad {+}\E{\mb{\epsilon}_t}{\lambda_t^e(\mb{\theta}_t) + \epsilon_t^e - (\lambda_t^e(\mb{\theta}_t)+\epsilon_t^r)^+|\mathcal{E}_2}\Prob{\mathcal{E}_2}\\
                   &\quad {+}\E{\mb{\epsilon}_t}{m^{i-1}_{t+1}(\mb{\lambda}_t(\mb{\theta}_t)+\mb{\epsilon}_t)|\mathcal{E}_3}\Prob{\mathcal{E}_3}\\
                   &{=} \E{\mb{\epsilon}_t}{M_{i}(\mb{\theta}_t,\mb{\epsilon}_t)},
\end{split}
\end{equation}
where we have used the price evolution equation we defined in \eqref{sys:coopt.price.evol} and the definition of $M(\mb{\theta}_t,\mb{\epsilon}_t)$ from \eqref{eq:coopt.def.uniform.M}. The proof is completed by backward induction over $t$ where \eqref{eq:coopt.uniform.cap.val} is obtained at $t=0$.
\qed

\section{Proof of Theorem~\ref{thm:coopt.uniform.indep}\label{app:coopt.pf.uniform.indep}}
With independent prices, we have $ \mb{\pi}_t=\mb{\epsilon}_t $. Moreover, remaining demand, $ d_t $, is the only element of state space. Therefore, the proposed form of the value function is automatically obtained since $m^{i}_{t}(\mb{\theta}_t)$ is no longer a function of $ \mb{\theta}_t $ and hence $m^{i}_{t}(\mb{\theta}_t)=\hat{m}^{i}_{t}$. This is because $M_i(\mb{\theta}_t,\mb{\epsilon}_t)$ is no longer a function of $ \mb{\theta}_t $, i.e. $M_i(\mb{\theta}_t,\mb{\epsilon}_t)=M_i(\mb{\epsilon}_t)$ and hence there would be no need to solve \eqref{eq:coopt.rec.m.eq}. Consequently, equations \eqref{eq:coopt.rec.uniform.cap}, \eqref{eq:coopt.def.uniform.M} and \eqref{eq:coopt.rec.m.eq} can be consolidated into a single recursion and conditional expected values can be approached directly. This is essentially what we establish in this proof.

First let us define $ \epsilon^a_t=\pi^a_t\triangleq \pi_t^{e}-(\pi_t^{r})^{+}$ as the effective price variable. Now, let us revisit the definition of $ M_i(\mb{\epsilon}_t) $, in this case:
\begin{equation}\label{eq:coopt.def.uniform.M.pf}
M_i(\mb{\epsilon}_{t})=\begin{cases}[@{}l@{\qquad}r@{}l@{}]
\hat{m}^{i}_{t+1} & \hat{m}^{i}_{t+1}\leq& \epsilon^a_t,\\
\epsilon^a_t & \hat{m}^{i-1}_{t+1}\leq&\epsilon^a_t < \hat{m}^{i}_{t{+}1},\\
\hat{m}^{i-1}_{t+1} & &\epsilon^a_t < \hat{m}^{i-1}_{t+1}.
\end{cases}
\end{equation}
Note that the three cases in the above definition corresponds to the three events $ \mathcal{E}_1 $, $ \mathcal{E}_2 $ and $ \mathcal{E}_3 $ defined in \eqref{eq:coopt.ev1}, \eqref{eq:coopt.ev2} and \eqref{eq:coopt.ev3} respectively. Now, we can use this notation to obtain a closed form for \eqref{eq:coopt.rec.m.eq}, which would essentially be \eqref{eq:coopt.rec.m.eq.indep}. Starting with \eqref{eq:coopt.rec.m.eq}:

\begin{equation}\label{eq:coopt.indep.pf.m}
\begin{split}
\hat{m}^{i}_{t}&=\E{\mb{\epsilon}_t}{M_i(\mb{\epsilon})}\\
&=\E{\mb{\epsilon}_t}{\hat{m}^{i}_{t+1}|\mathcal{E}_1}\Prob{\mathcal{E}_1} +\E{\mb{\epsilon}_t}{\epsilon^a_t|\mathcal{E}_2}\Prob{\mathcal{E}_2} +\E{\mb{\epsilon}_t}{\hat{m}^{i-1}_{t+1}|\mathcal{E}_3}\Prob{\mathcal{E}_3}\\
&=\hat{m}^{i}_{t+1}\Prob{\mathcal{E}_1} +\underbrace{\E{\mb{\epsilon}_t}{\epsilon^a_t|\mathcal{E}_2}\Prob{\mathcal{E}_2}}_{A} + \hat{m}^{i-1}_{t+1}\Prob{\mathcal{E}_3}\\
\end{split}
\end{equation}

Now, using integration by parts:
\begin{equation*}
\begin{split}
A&=\int_{\hat{m}^{i-1}_{t+1}}^{\hat{m}^{i}_{t+1}} \zeta\, dF^a_t(\zeta)\\
 &=  \hat{m}^{i}_{t+1}F^a_t(\hat{m}^{i}_{t+1})-\hat{m}^{i-1}_{t+1}F^a_t(\hat{m}^{i-1}_{t+1})-\int_{\hat{m}^{i-1}_{t+1}}^{\hat{m}^{i}_{t+1}} F^a_{t}(\zeta)\, d\zeta\\
 &= \hat{m}^{i}_{t+1}\Prob{\overline{\mathcal{E}}_1}-\hat{m}^{i-1}_{t+1}\Prob{\mathcal{E}_3}-G_t(\hat{m}^{i-1}_{t+1},\hat{m}^{i}_{t+1}),
\end{split}
\end{equation*}
where $\overline{\mathcal{E}}_1$ is the complement of event $\mathcal{E}_1$ and we have used definition \eqref{eq:coopt.def.G2}. Plugging back for $A$ in \eqref{eq:coopt.indep.pf.m}, we get:
\begin{equation*}
\begin{split}
\hat{m}^{i}_{t}
&=\hat{m}^{i}_{t+1}\Prob{\mathcal{E}_1} +\hat{m}^{i}_{t+1}\Prob{\overline{\mathcal{E}}_1}-\hat{m}^{i-1}_{t+1}\Prob{\mathcal{E}_3}-G_t(\hat{m}^{i-1}_{t+1},\hat{m}^{i}_{t+1}) + \hat{m}^{i-1}_{t+1}\Prob{\mathcal{E}_3}\\
&= \hat{m}^{i}_{t+1}-G_t(\hat{m}^{i-1}_{t+1},\hat{m}^{i}_{t+1}),
\end{split}
\end{equation*}
which is the desired result.
\qed

\label{lastpage}

\end{document}